\documentclass[12pt]{amsart}
\usepackage{amssymb}
\usepackage{amsfonts}
\usepackage{amsbsy}
\usepackage{latexsym}
\newtheorem{theorem}{Theorem}
\newtheorem{lemma}[theorem]{Lemma}
\newtheorem{corollary}[theorem]{Corollary}

\theoremstyle{definition}
\newtheorem{definition}[theorem]{Definition}

\begin{document}
\title {Jordan Triple Elementary Maps on Rings}
\author{Wu Jing}
\address{Department of Mathematics and Computer Science, Fayetteville State University,  Fayetteville, NC 28301}
 \email{wjing@uncfsu.edu}

\date{May 25, 2007}
\begin{abstract}
We prove that Jordan triple elementary surjective maps on unital rings containing a nontrivial idempotent are automatically additive. \end{abstract}
\maketitle

The first result about the additivity of maps on rings was given by Martindale III in an  excellent paper \cite {ma695}. He established a condition on a ring $\mathcal {R}$ such that every multiplicative bijective map on $\mathcal {R}$ is additive. More precisely, he proved the following theorem.
\begin{theorem}(\cite{ma695}) Let $\mathcal {R}$ be a ring containing a family $\{ e_{\alpha }:\alpha \in \Lambda \} $ of idempotents which satisfies:

(i) $x\mathcal {R}=\{ 0\} $ implies $x=0$;

(ii) If $e_{\alpha }\mathcal {R}x=\{ 0\} $ for each $\alpha \in \Lambda $, then $x=0$;

(iii) For each $\alpha \in \Lambda $, $e_{\alpha }xe_{\alpha }\mathcal {R}(1-e_{\alpha })=\{ 0\} $ implies $e_{\alpha }xe_{\alpha }=0$.

Then any multiplicative bijective map from $\mathcal {R}$ onto an arbitrary ring $\mathcal {R}^{\prime }$ is additive.
\end{theorem}

As a corollary, every multiplicative bijective map from a prime ring containing a nontrivial idempotent onto an arbitrary ring is necessarily additive.

During the last decade, many mathematicians devoted to study the additivity of maps on rings as well as operator algebras.  In this paper we continue to investigate the additivity of Jordan triple elementary maps on rings.

   We first define Jordan triple
elementary maps as follows.
\begin{definition}
Let $\mathcal {R}$ and $\mathcal {R}^{\prime}$ be two rings, and let
$M\colon \mathcal {R}\rightarrow \mathcal {R}^{\prime}$ and
$M^*\colon \mathcal {R}^{\prime}\rightarrow \mathcal {R}$ be two
maps.  Call the ordered pair $(M, M^*)$ a \textit {Jordan triple elementary
map} of $\mathcal {R}\times \mathcal {R}^{\prime}$ if
\begin{displaymath}
 \left\{ \begin{array}{ll}
 M(aM^*(x)b+bM^*(x)a)=M(a)xM(b)+M(b)xM(a),\\
 M^*(xM(a)y+yM(a)x)=M^*(x)aM^*(y)+M^*(y)aM^*(x)
 \end{array}\right.
\end{displaymath}
 for all $a, b\in \mathcal {R}$ and $x, y\in \mathcal {R}^{\prime}$.
\end{definition}

Let us now introduce some definitions and results.  Let $\mathcal {R}$
be a ring, if $a\mathcal {R}b=\{ 0\} $ implies either $a=0$ or
$b=0$, then $\mathcal {R}$ is called a \textit{prime} ring. Ring $\mathcal {R}$ is said to be \textit {$2$-torsion free} if $2a=0$ implies $a=0$.

Suppose that $\mathcal {R}$ is a ring containing  a nontrivial
idempotent $e_1$. Let $e_2=1-e_1$.  Note that $\mathcal {R}$ need not have an identity element. We set $\mathcal {R}_{ij}=e_i\mathcal
{R}e_j$, for $i, j=1, 2$. Then we may write $\mathcal {R}=\mathcal
{R} _{11}\oplus \mathcal {R} _{12}\oplus \mathcal {R}_{21}\oplus
\mathcal {R} _{22}$. It should be mentioned here that this
significant idea is due to Martinadale (\cite {ma695}) which has become a
key tool in dealing with a large number of maps on some rings and
algebras. In what follows, $a_{ij}$ will denote that  $a_{ij}\in
\mathcal {R}_{ij}$ ($1\leq i, j\leq 2$).

We denote by $B(X)$ the algebra of all linear bounded operators on a
Banach space $X$. A subalgebra of
$B(X)$ is called a \textit {standard operator algebra} if it
contains all finite rank operators in $B(X)$.

Let's state the  main result of this paper.
\begin{theorem}\label{theorem}
Let $\mathcal {R}$ and $\mathcal {R}^{\prime }$ be two rings.
Suppose that $\mathcal {R}$ is a $2$-torsion free unital ring containing   a
nontrivial idempotent $e_1$ and satisfies

(P) $e_iae_j\mathcal {R}e_k=\{ 0\} $, or $e_k\mathcal {R}e_iae_j=\{
0\} $ implies $e_iae_j=0$ ($1\leq i, j, k\leq 2$), where
$e_2=1-e_1$;

Suppose that $M\colon {\mathcal R}\rightarrow {\mathcal R}^{\prime
}$ and $M^*\colon {\mathcal R}^{\prime }\rightarrow {\mathcal R}$
are surjective maps such that
\begin{displaymath}
 \left\{ \begin{array}{ll}
  M(aM^*(x)b+bM^*(x)a)=M(a)xM(b)+M(b)xM(a),\\
 M^*(xM(a)y+yM(a)x)=M^*(x)aM^*(y)+M^*(y)aM^*(x)
 \end{array}\right.
\end{displaymath}
 for all $a, b\in \mathcal {R}$ and $x, y\in \mathcal {R}^{\prime}$. Then both $M$ and $M^*$ are additive. \end{theorem}

To prove this theorem, let's introduce  a series of lemmas.  We
begin with

\begin{lemma} $M(0)=0$ and $M^*(0)=0$.
\end{lemma}
\begin{proof} We have
$M(0)=M(0M^*(0)0+0M^*(0)0)=M(0)0M(0)+M(0)0M^*(0)=0$.

Similarly, $M^*(0)=M^*(0M(0)0+0M(0)0)=M^*(0)0M^*(0)+M^*(0)0M^*(0)=0$.
\end{proof}

The proof of the following lemma is straightforward.

\begin{lemma}\label{lu}
Let $a=a_{11}+a_{12}+a_{21}+a_{22}\in \mathcal {R}$.

(i) If $a_{ij}t_{jk}=0$ for each $t_{jk}\in \mathcal {R}_{jk} $
($1\leq i, j, k\leq 2$), then $a_{ij}=0$.

Dually, if $t_{ki}a_{ij}=0$ for each  $t_{ki}\in \mathcal {R}_{ki} $
($1\leq i, j, k\leq 2$), then $a_{ij}=0$.

(ii) If $a_{ij}t+ta_{ij}\in \mathcal {R}_{ij} $ for every $a_{ij}\in
\mathcal {R}_{ij}$ ($1\leq i, j\leq 2$), then $t_{ji}=0$

 (iii) If $a_{ii}t+ta_{ii}=0$ for every $a_{ii}\in \mathcal {R}_{ii}$ ($i=1, 2$), then $a_{ii}=0$;

(iv) If $a_{jj}t+ta_{jj}\in \mathcal{R}_{ij}$ for every $a_{jj}\in
\mathcal {R}_{jj}$ ($1\leq i\not =j\leq j$), then $t_{ji}=0$ and
$a_{jj}=0$.

  Dually, if  $a_{jj}t+ta_{jj}\in \mathcal{R}_{ji}$ for every $a_{jj}\in \mathcal {R}_{jj}$ ($1\leq i\not =j\leq j$), then $s_{ji}=0$ and $a_{jj}=0$.

\end{lemma}

 \begin{lemma}
$M$ and $M^*$ are injective.
\end{lemma}
\begin{proof} We first prove that $M$ is injective. Suppose that $M(a)=M(b)$ for $a, b\in\mathcal {R}$.
We write  $a=a_{11}+a_{12}+a_{21}+a_{22}$ and $b=b_{11}+b_{12}+b_{21}+b_{22}$.

For arbitrary $t_{ij}\in \mathcal {R}_{ij}$ and $s_{kl}\in \mathcal {R}_{kl}$ ($1\leq i, j, k, l\leq 2$),
there exist $x(i, j)\in \mathcal {R}^{\prime}$ and $y(k, l)\in \mathcal {R}^{\prime}$ such that
$M^*(x(i, j))=t_{ij}$  and
$M^*(y(k, l))=s_{kl}$ since $M^*$ is surjective.  We now compute
\begin{eqnarray*}
& &s_{kl}at_{ij}+t_{ij}as_{kl}\\
&=&M^*(y(k,l))aM^*(x(i,j))+M^*(x(i,j))aM^*(y(k,l))\\
&=&M^*(y(k,l)M(a)x(i,j)+x(i,j)M(a)y(k,l))\\
&=&M^*(y(k,l)M(b)x(i,j)+x(i,j)M(b)y(k,l))\\
&=&M^*(y(k,l))bM^*(x(i,j))+M^*(x(i,j))bM^*(y(k,l))\\
&=&s_{kl}bt_{ij}+t_{ij}bs_{kl}.
\end{eqnarray*}

 Letting $i=j=k=1$ and $l=2$, we get $s_{12}at_{11}+t_{11}as_{12}=s_{12}bt_{11}+t_{11}bs_{12}$, which implies that $a_{11}=b_{11}$ and $a_{21}=b_{21}$.

 Similarly, we can get $a_{22}=b_{22}$ and $a_{12}=b_{12}$ by considering $i=k=l=2$ and $j=1$. Therefore, we arrive at $a=b$.

To show that $M^*$ is injective,  let $x, y\in \mathcal
{R}^{\prime}$ such that $M^*(x)=M^*(y)$.
For any $t_{ij}\in \mathcal {R}_{ij}$ and $s_{kl}\in \mathcal {R}_{kl}$, by the surjectivity of
$M^*M$, there are  $c(i, j)\in \mathcal {R}$  and  $d(k, l)\in \mathcal {R}$ such that $M^*M(c(i,
j))=t_{ij}$ and $M^*M(d(k, l))=s_{kl}$.

We consider
\begin{eqnarray*}
& &t_{ij}M^{-1}(x)s_{kl}+s_{kl}M^{-1}(x)t_{ij}\\
&=&M^*M(c(i, j))M^{-1}(x)M^*M(d(k, l))+M^*M(d(k, l))M^{-1}(x)M^*M(c(i, j))\\
&=&M^*(M(c(i, j))xM(d(k, l))+M(d(k, l))xM(c(i, j)))\\
&=&M^*M(c(i, j)M^*(x)d(k, l)+d(k, l)M^*(x)c(i, j))\\
&=&M^*M(c(i, j)M^*(y)d(k, l)+d(k, l)M^*(y)c(i, j))\\
&=&M^*(M(c(i, j))yM(d(k, l))+M(d(k, l))yM(c(i, j)))\\
&=&M^*M(c(i, j))M^{-1}(y)M^*M(d(k, l))+M^*M(d(k, l))M^{-1}(y)M^*M(c(i, j))\\
&=&t_{ij}M^{-1}(x)s_{kl}+s_{kl}M^{-1}(x)t_{ij}.
\end{eqnarray*}

With the same argument used in showing the injectivity of $M$, one can easily get  $M^{-1}(x)=M^{-1}(y)$.
Therefore, $x=y$, this completes the proof.
\end{proof}

 \begin{lemma}\label{inverse}
The pair $(M^{*^{-1}}, M^{-1})$ is a Jordan elementary map on
$\mathcal {R}\times \mathcal {R}^{\prime}$. That is,
\begin{displaymath}
 \left\{ \begin{array}{ll}
 M^{*^{-1}}(aM^{-1}(x)b+bM^{-1}(x)a)=M^{*^{-1}}(a)xM^{*^{-1}}(b)+M^{*^{-1}}(b)xM^{*^{-1}}(a),\\
 M^{-1}(xM^{*^{-1}}(a)y+yM^{*^{-1}}(a)x)=M^{-1}(x)aM^{-1}(y)+M^{-1}(y)aM^{-1}(x)
\end{array}\right.
\end{displaymath}
for all $a, b\in \mathcal {R}$ and  $x, y\in \mathcal {R}^{\prime }$.
\end{lemma}
\begin{proof}
From
\begin{eqnarray*}
& &M^*(M^{*^{-1}}(a)xM^{*^{-1}}(b)+M^{*^{-1}}(b)xM^{*^{-1}}(a))\\
&=&M^*(M^{*^{-1}}(a)MM^{-1}(x)M^{*^{-1}}(b)+M^{*^{-1}}(b)MM^{-1}(x)M^{*^{-1}}(a))\\
&=&M^*M^{*^{-1}}(a)M^{-1}(x)M^*M^{*^{-1}}(b)+M^*M^{*^{-1}}(b)M^{-1}(x)M^*M^{*^{-1}}(a)\\
&=&aM^{-1}(x)b+bM^{-1}(x)a,
\end{eqnarray*}
we can directly get $M^{*^{-1}}(aM^{-1}(x)b+bM^{-1}(x)a)=M^{*^{-1}}(a)xM^{*^{-1}}(b)+M^{*^{-1}}(b)xM^{*^{-1}}(a)$. The rest of the proof follows similarly.
\end{proof}

The following result will be used frequently in this note.
\begin{lemma}\label{add}
Let $a, b, c\in \mathcal {R}$ such that $M(c)=M(a)+M(b)$. Then
$$M^{*^{-1}}(tcs+sct)=M^{*^{-1}}(tas+sat)+M^{*^{-1}}(tbs+sbt)$$
for all $t, s\in \mathcal {R}$
\end{lemma}
\begin{proof} For every $t, s\in \mathcal {R}$, by Lemma \ref{inverse}, we have
\begin{eqnarray*}
& &M^{*^{-1}}(tcs+sct)\\&=&M^{*^{-1}}(tM^{-1}M(c)s+sM^{-1}M(c)t)\\
&=&M^{*^{-1}}(t)M(c)M^{*^{-1}}(s)+M^{*^{-1}}(s)M(c)M^{*^{-1}}(t)\\
&=&M^{*^{-1}}(t)(M(a)+M(b))M^{*^{-1}}(s)+M^{*^{-1}}(s)(M(a)+M(b))M^{*^{-1}}(t)\\
&=&(M^{*^{-1}}(t)M(a)M^{*^{-1}}(s)+M^{*^{-1}}(s)M(a)M^{*^{-1}}(t))\\
& &+(M^{*^{-1}}(t)M(b)M^{*^{-1}}(s)+M^{*^{-1}}(s)M(b)M^{*^{-1}}(t))\\
&=&M^{*^{-1}}(tas+sat)+M^{*^{-1}}(tbs+sbt).
\end{eqnarray*}
\end{proof}

\begin{lemma}\label{lemmaiiij}
Let $a_{ii}\in \mathcal {R}_{ii}$ and $b_{ij}\in \mathcal {R}_{ij}$,
$1\leq i\not =j\leq 2$, then

(i) $M(a_{ii}+b_{ij})=M(a_{ii})+M(b_{ij})$;

(ii)
$M^{*^{-1}}(a_{ii}+b_{ij})=M^{*^{-1}}(a_{ii})+M^{*^{-1}}(b_{ij})$.
\end{lemma}
\begin{proof}
Let $c\in \mathcal {R}$ be chosen such that $M(c)=M(a_{ii})+M(b_{ij})$. For arbitrary $t_{ij}\in \mathcal {R}_{ij}$ and $d_{ii}\in \mathcal {R}_{ii}$, by Lemma \ref{add}, we have
\begin{eqnarray*}
& &M^{*^{-1}}(t_{ij}cd_{ii}+d_{ii}ct_{ij})\\
&=&M^{*^{-1}}(t_{ij}a_{ii}d_{ii}+d_{ii}a_{ii}t_{ij})+M^{*^{-1}}(t_{ij}b_{ij}d_{ii}+d_{ii}b_{ij}t_{ij})\\
&=&M^{*^{-1}}(d_{ii}a_{ii}t_{ij}).
\end{eqnarray*}
Thus  $t_{ij}cd_{ii}+d_{ii}ct_{ij}=d_{ii}a_{ii}t_{ij}$, i.e., $t_{ij}c_{ji}d_{ii}+d_{ii}c_{ii}t_{ij}=d_{ii}a_{ii}t_{ij}$.  By Lemma \ref{lu}, we have $c_{ji}=0$ and $c_{ii}=a_{ii}$.

  Now for any $t_{jj}\in \mathcal {R}_{jj}$ and $d_{ji}\in \mathcal {R}_{ji}$, using Lemma \ref{add}, we have
  \begin{eqnarray*}
 & &M^{*^{-1}}(t_{jj}cd_{ji}+d_{ji}ct_{jj})\\
&=&M^{*^{-1}}(t_{jj}a_{ii}d_{ji}+d_{ji}a_{ii}t_{jj})+M^{*^{-1}}(t_{jj}b_{ij}d_{ji}+d_{ji}b_{ij}t_{jj})\\
&=&M^{*^{-1}}(d_{ji}b_{ij}t_{jj}).
\end{eqnarray*}
It follows that $t_{jj}cd_{ji}+d_{ji}ct_{jj}=d_{ji}b_{ij}t_{jj}$, and so $t_{jj}c_{jj}d_{ji}+d_{ji}b_{ij}t_{jj}=d_{ji}b_{ij}t_{jj}$, which leads to $c_{jj}=0$ and $c_{ij}=b_{ij}$. Therefore $c=a_{ii}+b_{ij}$, as desired.

 By Lemma \ref{inverse} we can infer that (ii) holds.
\end{proof}

Similarly, we can get the following result.
\begin{lemma}\label{lemmaiiji}
Let $a_{ii}\in \mathcal {R}_{ii}$ and $b_{ji}\in \mathcal {R}_{ji}$,
$1\leq i\not =j\leq 2$, then

(i) $M(a_{ii}+b_{ji})=M(a_{ii})+M(b_{ji})$;

(ii)
$M^{*^{-1}}(a_{ii}+b_{ji})=M^{*^{-1}}(a_{ii})+M^{*^{-1}}(b_{ji})$.
\end{lemma}

\begin{lemma} \label{lemma121222}(i) $M(a_{12}+b_{12}c_{22})=M(a_{12})+M(b_{12}c_{22})$;

(ii)
$M^{*^{-1}}(a_{12}+b_{12}c_{22})=M^{*^{-1}}(a_{12})+M^{*^{-1}}(b_{12}c_{22})$;

(iii)  $M(a_{21}+b_{22}c_{21})=M(a_{21})+M(b_{22}c_{21})$;

(iv)
$M^{*^{-1}}(a_{21}+b_{22}c_{21})=M^{*^{-1}}(a_{21})+M^{*^{-1}}(b_{22}c_{21})$.
\end{lemma}

\begin{proof} We only prove (i) and (iii).

  Note that $a_{12}+b_{12}c_{22}=(e_1+b_{12})(a_{12}+c_{22})e_2+e_2(a_{12}+c_{22})(e_1+b_{12})$. We now compute
\begin{eqnarray*}
& &M(a_{12}+b_{12}c_{22})\\
&=&M((e_1+b_{12})(a_{12}+c_{22})e_2+e_2(a_{12}+c_{22})(e_1+b_{12}))\\
&=&M((e_1+b_{12})M^*M^{*^{-1}}(a_{12}+c_{22})e_2+e_2M^*M^{*^{-1}}(a_{12}+c_{22})(e_1+b_{12}))\\
&=&M(e_1+b_{12})M^{*^{-1}}(a_{12})M(e_2)+M(e_1+b_{12})M^{*^{-1}}(c_{22})M(e_2)\\
& &+M(e_2)M^{*^{-1}}(a_{12})M(e_1+b_{12})+M(e_2)M^{*^{-1}}(c_{22})M(e_1+b_{12})\\
&=&M((e_1+b_{12})a_{12}e_2+e_2a_{12}(e_1+b_{12}))+M((e_1+b_{12})c_{22}e_2+e_2c_{22}(e_1+b_{12}))\\
&=&M(a_{12})+M(b_{12}c_{22}).
\end{eqnarray*}
  Similarly, we can get  $M(a_{21}+c_{22}b_{21})=M(a_{21})+M(c_{22}b_{21})$ from the fact that $a_{21}+c_{22}b_{21}=(e_1+b_{21})(a_{21}+c_{22})e_2+e_2(a_{21}+c_{22})(e_1+b_{21})$.

(ii) and (iv) follow from (i) and (iii) respectively by Lemma
\ref{inverse}.
 \end{proof}

 \begin{lemma} \label{lemma12} For any $a_{12}, b_{12}\in \mathcal {R}_{12}$, we have

 (i) $M(a_{12}+b_{12})=M(a_{12})+M(b_{12})$;

(ii)
$M^{*^{-1}}(a_{12}+b_{12})=M^{*^{-1}}(a_{12})+M^{*^{-1}}(b_{12})$.
 \end{lemma}
 \begin{proof}
We only show (i). We pick  $c=c_{11}+c_{12}+c_{21}+c_{22}\in
\mathcal {R}$ such that $M(c)=M(a_{12})+M(b_{12})$. For any
$t_{11}\in \mathcal {R}_{11}$ and $s_{22}\in \mathcal {R}_{22}$, we have
\begin{eqnarray*}
& &M^{*^{-1}}(t_{11}cs_{22}+s_{22}ct_{11})\\
&=&M^{*^{-1}}(t_{11}a_{12}s_{22}+s_{22}a_{12}t_{11})+M^{*^{-1}}(t_{11}b_{12}s_{22}+s_{22}b_{12}t_{11})\\
&=&M^{*^{-1}}(t_{11}a_{12})+M^{*^{-1}}(t_{11}b_{12})\\
&=&M^{*^{-1}}(t_{11}a_{12}s_{22}+t_{11}b_{12}s_{22}).
\end{eqnarray*}

Note that   we apply Lemma \ref{add} in the first equality and Lemma \ref{lemma121222} in the last equality.

Therefore we have  $t_{11}cs_{22}+s_{22}ct_{11}=t_{11}a_{12}s_{22}+t_{11}b_{12}s_{22}$. Consequently,
$$ t_{11}c_{21}s_{22}+s_{22}c_{21}t_{11}=t_{11}a_{12}s_{22}+t_{11}b_{12}s_{22}.$$
It follows that $c_{21}=0$ and $c_{12}=a_{12}+b_{12}$.

   To complete the proof it remains to show that $c_{11}=c_{22}=0$. For arbitrary $t_{21}\in \mathcal {R}_{21}$ and $s_{12}\in \mathcal {R}_{12}$, by Lemma \ref{add}, we  compute
\begin{eqnarray*}
& & M^{*^{-1}}(t_{21}cs_{12}+s_{12}ct_{21})\\
&=&M^{*^{-1}}(t_{21}a_{12}s_{12}+s_{12}a_{12}t_{21})+M^{*^{-1}}(t_{21}b_{12}s_{12}+s_{12}b_{12}t_{21})\\
&=&0.
\end{eqnarray*}
This yields that $t_{21}cs_{12}+s_{12}ct_{21}=0$. Furthermore, $t_{21}c_{11}s_{12}+s_{12}c_{22}t_{21}=0$.
By Lemma \ref{lu}, we see that $c_{11}=c_{22}=0$.
  \end{proof}

  \begin{lemma}\label{lemma21}
  The following hold.

  (i)   $M(a_{21}+b_{21})=M(a_{21})+M(b_{21})$;

 (ii)  $M^{*^{-1}}(a_{21}+b_{21})=M^{*^{-1}}(a_{21})+M^{*^{-1}}(b_{21})$.
 \end{lemma}
 \begin{proof} Let  $c=c_{11}+c_{12}+c_{21}+c_{22}\in \mathcal {R}$ be chosen such that $M(a_{21})+M(b_{21})=M(c)$.

 For any $t_{22}\in \mathcal {R}_{22}$ and $s_{11}\in \mathcal {R}_{11}$, using Lemma \ref{add} and Lemma \ref{lemma121222}, we have
 \begin{eqnarray*}
 & &M^{*^{-1}}(t_{22}cs_{11}+s_{11}ct_{22})\\
 &=&M^{*^{-1}}(t_{22}a_{21}s_{11}+s_{11}a_{21}t_{22})+M^{*^{-1}}(t_{22}b_{21}s_{11}+s_{11}b_{21}t_{22})\\
 &=&M^{*^{-1}}(t_{22}a_{21}s_{11})+M^{*^{-1}}(t_{22}b_{21}s_{11})\\
 &=&M^{*^{-1}}(t_{22}a_{21}s_{11}+t_{22}b_{21}s_{11})
 \end{eqnarray*}
 which implies that $t_{22}cs_{11}+s_{11}ct_{22}=t_{22}a_{21}s_{11}+t_{22}b_{21}s_{11}$, and so
  $$ t_{22}cs_{11}+s_{11}ct_{22}=t_{22}a_{21}s_{11}+t_{22}b_{21}s_{11}.$$
 It follows that $$t_{22}c_{21}s_{11}+s_{11}c_{12}t_{22}=t_{22}a_{21}s_{11}+t_{22}b_{21}s_{11}.$$ Therefore we can infer that  $c_{12}=0$ and $c_{21}=a_{21}+b_{21}$.

We now show that $C_{11}=c_{22}=0$. To this aim, for any $t_{12}\in \mathcal {R}_{12}$ and $s_{21}\in \mathcal {R}_{21}$, let's consider
\begin{eqnarray*}
& &M^{*^{-1}}(t_{12}cs_{21}+s_{21}ct_{12})\\
 &=&M^{*^{-1}}(t_{12}a_{21}s_{21}+s_{21}a_{21}t_{12})+M^{*^{-1}}(t_{12}b_{21}s_{21}+s_{21}b_{21}t_{12})\\
 &=&0.
\end{eqnarray*}
It follows that $t_{12}cs_{21}+s_{21}ct_{12}=0$, and so $t_{12}c_{22}s_{21}+s_{21}c_{11}t_{12}=0$. Hence $c_{11}=c_{22}=0$. The proof is complete.
   \end{proof}

   \begin{lemma} \label{lemma11}For arbitrary $a_{11}, b_{11}\in \mathcal {R}_{11}$, the following are true.

  (i)   $M(a_{11}+b_{11})=M(a_{11})+M(b_{11})$;

 (ii)  $M^{*^{-1}}(a_{11}+b_{11})=M^{*^{-1}}(a_{11})+M^{*^{-1}}(b_{11})$.
 \end{lemma}
 \begin{proof} We only prove (i). Pick  $c=c_{11}+c_{12}+c_{21}+c_{22}\in \mathcal {R}$  such that $M(c)=M(a_{11})+M(b_{11})$.

 For any $t_{22}\in \mathcal {R}_{22}$ and $s_{ij}\in \mathcal {R}_{ij}$, by Lemma \ref{add}, we have
$$ M^{*^{-1}}(t_{22}cs_{ij}+s_{ij}ct_{22})=M^{*^{-1}}(t_{22}a_{11}s_{ij}+s_{ij}a_{11}t_{22})+M^{*^{-1}}(t_{22}b_{11}s_{ij}+s_{ij}b_{11}t_{22})=0.$$
 This implies that
 \begin{equation}\label{1a}
 t_{22}cs_{ij}+s_{ij}ct_{22}=0.
 \end{equation}

 Letting $i=j=1$ in the above equality, we get $t_{22}c_{21}s_{11}+s_{11}c_{12}t_{22}=0$, it follows that $c_{21}=c_{12}=0$.

 If we let $i=2$ and $j=1$ in equality (\ref{1a}), then we get $t_{22}c_{22}s_{21}+s_{21}c_{12}t_{22}=0$. Therefore $c_{22}=0$.

 To complete the proof it remain to show that $c_{11}=a_{11}+b_{11}$. For arbitrary $t_{12}\in \mathcal {R}_{12}$ and $s_{11}\in \mathcal {R}_{11}$. We compute
 \begin{eqnarray*}
 & &M^{*^{-1}}(t_{12}cs_{11}+s_{11}ct_{12})\\
 &=&M^{*^{-1}}(t_{12}a_{11}s_{11}+s_{11}a_{11}t_{12})+M^{*^{-1}}(t_{12}b_{11}s_{11}+s_{11}b_{11}t_{12})\\
 &=&M^{*^{-1}}(s_{11}a_{11}t_{12})+M^{*^{-1}}(s_{11}b_{11}t_{12})\\
 &=&M^{*^{-1}}(s_{11}a_{11}t_{12}+s_{11}b_{11}t_{12}).
 \end{eqnarray*}
 It follows that $t_{12}cs_{11}+s_{11}ct_{12}=s_{11}a_{11}t_{12}+s_{11}b_{11}t_{12}$, and so  $t_{12}c_{21}s_{11}+s_{11}c_{11}t_{12}=s_{11}a_{11}t_{12}+s_{11}b_{11}t_{12}$. By Lemma \ref{lu}, we arrive at  $c_{11}=a_{11}+b_{11}$.
 \end{proof}
 Similarly, we have
 \begin{lemma} \label{lemma22}For arbitrary $a_{22}, b_{22}\in \mathcal {R}_{22}$, we have

  (i)   $M(a_{22}+b_{22})=M(a_{22})+M(b_{22})$;

 (ii)  $M^{*^{-1}}(a_{22}+b_{22})=M^{*^{-1}}(a_{22})+M^{*^{-1}}(b_{22})$.
 \end{lemma}

  \begin{lemma}\label{lemma1122}
 For arbitrary $a_{11}\in \mathcal {R}_{11}$ and $b_{22}\in \mathcal {R}_{22}$, the following hold.

 (i) $M(a_{11}+b_{22})=M(a_{11})+M(b_{22})$;

 (ii) $M^{*^{-1}}(a_{11}+b_{22})=M^{*^{-1}}(a_{11})+M^{*^{-1}}(b_{22})$.
 \end{lemma}
 \begin {proof}
 We only prove (i). Let  $c=c_{11}+c_{12}+c_{21}+c_{22}$ be an element of  $\mathcal {R}$ satisfying $M(c)=M(a_{11})+M(a_{22})$.

  For any $t_{11}\in \mathcal {R}_{11}$ and $s_{21}\in \mathcal {R}_{21}$, we compute
  \begin{eqnarray*}
  & &M^{*^{-1}}(t_{11}cs_{21}+s_{21}ct_{11})\\
  &=&M^{*^{-1}}(t_{11}a_{11}s_{21}+s_{21}a_{11}t_{11})+M^{*^{-1}}(t_{11}b_{22}s_{21}+s_{21}b_{22}t_{11})\\
 &=&M^{*^{-1}}(s_{21}a_{11}t_{11}).
 \end{eqnarray*}
 This implies that $t_{11}cs_{21}+s_{21}ct_{11}=s_{21}a_{11}t_{11}$.  It follows that $t_{11}c_{12}s_{21}+s_{21}c_{11}t_{11}=s_{21}a_{11}t_{11}$, and so $c_{12}=0$ and $c_{11}=a_{11}$.

 To complete the proof, we need to show that $c_{22}=b_{22}$ and $c_{21}=0$. For any $t_{22}\in \mathcal {R}_{22}$ and $s_{12}\in \mathcal {R}_{12}$, we obtain
 \begin{eqnarray*}
 & &M^{*^{-1}}(t_{22}cs_{12}+s_{12}ct_{22})\\
 &=&M^{*^{-1}}(t_{22}a_{11}s_{12}+s_{12}a_{11}t_{22})+M^{*^{-1}}(t_{22}b_{22}s_{12}+s_{12}b_{22}t_{22})\\
 &=&M^{*^{-1}}(s_{12}b_{22}t_{22}).
 \end{eqnarray*}
  It follows that $t_{22}cs_{12}+s_{12}ct_{22}=s_{12}b_{22}t_{22}$, which leads to $t_{22}c_{21}s_{12}+s_{12}c_{22}t_{22}=s_{12}b_{22}t_{22}$, and so $c_{22}=b_{22}$ and $c_{21}=0$. The proof is done.

 \end{proof}
  \begin{lemma}
 For arbitrary $a_{12}\in \mathcal {R}_{12}$ and $b_{21}\in \mathcal {R}_{21}$, we have

 (i) $M(a_{12}+b_{21})=M(a_{12})+M(b_{21})$;

 (ii) $M^{*^{-1}}(a_{12}+b_{21})=M^{*^{-1}}(a_{12})+M^{*^{-1}}(b_{21})$.
 \end{lemma}
 \begin {proof}
Suppose that  $M(c)=M(a_{12})+M(a_{21})$ for some $c=c_{11}+c_{12}+c_{21}+c_{22}\in \mathcal {R}$.

 Now for arbitrary $t_{12}\in \mathcal {R}_{12}$ and $s_{11}\in \mathcal {R}_{11}$, we have
 \begin{eqnarray*}
 & &M^{*^{-1}}(t_{12}cs_{11}+s_{11}ct_{12})\\
 &=&M^{*^{-1}}(t_{12}a_{12}s_{11}+s_{11}a_{12}t_{12})+M^{*^{-1}}(t_{12}b_{21}s_{11}+s_{11}b_{21}t_{12})\\
 &=&M^{*^{-1}}(t_{12}b_{21}s_{11}).
 \end{eqnarray*}
 Therefore
 $$ t_{12}cs_{11}+s_{11}ct_{12}=t_{12}b_{21}s_{11},$$
 i. e., $t_{12}c_{21}s_{11}+s_{11}c_{11}t_{12}=t_{12}b_{21}s_{11}$. This implies that $c_{21}=b_{21}$ and $c_{11}=0$ by Lemma \ref{lu}.

 We now show that $C_{12}=a_{12}$ and $c_{22}=0$. For any $t_{21}\in \mathcal {R}_{21}$ and $s_{22}\in \mathcal {R}_{22}$, we obtain
 \begin{eqnarray*}
 & &M^{*^{-1}}(t_{21}cs_{22}+s_{22}ct_{21})\\
 &=&M^{*^{-1}}(t_{21}a_{12}s_{22}+s_{22}a_{12}t_{21})+M^{*^{-1}}(t_{21}b_{21}s_{22}+s_{22}b_{21}t_{21})\\
 &=&M^{*^{-1}}(t_{21}a_{12}s_{22}).
 \end{eqnarray*}
 Then we get $$t_{21}cs_{22}+s_{22}ct_{21}=t_{21}a_{12}s_{22},$$
  that is $$t_{21}c_{12}s_{22}+s_{22}c_{22}t_{21}=t_{21}a_{12}s_{22},$$
  which implies $c_{12}=a_{12}$ and $c_{22}=0$.

 \end{proof}

  \begin{lemma} \label{lemma111221}For any $a_{11}\in \mathcal {R}_{11}$, $b_{12}\in \mathcal {R}_{12}$, and $c_{21}\in \mathcal {R}_{21}$, we have

  (i) $M(a_{11}+b_{12}+c_{21})=M(a_{11})+M(b_{12})+M(c_{21})$;

  (ii) $M^{*^{-1}}(a_{11}+b_{12}+c_{21})=M^{*^{-1}}(a_{11})+M^{*^{-1}}(b_{12})+M^{*^{-1}}(c_{21})$.

 \end{lemma}
 \begin{proof}
 We choose  $d=d_{11}+d_{12}+d_{21}+d_{22}\in \mathcal {R}$ such that $M(d)=M(a_{11})+M(b_{12})+M(c_{21})$. By Lemma \ref{lemmaiiij} and Lemma \ref{lemmaiiji}, we have

\begin{equation}\label{ee}
M(d)=M(a_{11}+b_{12})+M(c_{21})
\end{equation}
and
\begin{equation}\label{ff}
M(d)=M(a_{11}+c_{21})+M(b_{12}).
\end{equation}
For any $t_{21}\in \mathcal {R}_{21}$ and $s_{12}\in \mathcal {R}_{12}$, by Lemma \ref{add} and
equation (\ref{ee}), we have
\begin{eqnarray*}
& &M^{*^{-1}}(t_{21}ds_{12}+s_{12}dt_{21})\\
&=&M(^{*^{-1}}t_{21}(a_{11}+b_{12})s_{12}+s_{12}(a_{11}+b_{12})t_{21})+M^{*^{-1}}(t_{21}c_{21}s_{12}+s_{12}c_{21}t_{21})\\
&=&M^{*^{-1}}(t_{21}a_{11}s_{12}),
\end{eqnarray*}
which yields that
$$
t_{21}ds_{12}+s_{12}dt_{21}=t_{21}a_{11}s_{12}.$$
Furthermore, $$t_{21}d_{11}s_{12}+s_{12}d_{22}t_{21}=t_{21}a_{11}s_{12}.$$ Therefore $d_{11}=a_{11}$ and $d_{22}=0$.

  In order to complete the proof, we need to show that $d_{22}=0$ and $d_{21}=c_{21}$.
 For arbitrary
$t_{22}\in \mathcal {R}_{22}$ and $s_{12}\in \mathcal {R}_{12}$, using Lemma \ref{add} and equality
(\ref{ff}), we have
\begin{eqnarray*}
& &M^{*^{-1}}(t_{22}ds_{12}+s_{12}dt_{22})\\
&=&M^{*^{-1}}(t_{22}(a_{11}+c_{21})s_{12}+s_{12}(a_{11}+c_{21})t_{22})+M^{*^{-1}}(t_{22}b_{12}s_{12}+s_{12}b_{12}t_{22})\\
&=&M^{*^{-1}}(t_{22}c_{21}s_{12}).
\end{eqnarray*}
This leads to
 $$t_{22}ds_{12}+s_{12}dt_{22}=t_{22}c_{21}s_{12}.$$
 Then we get
$$t_{22}d_{21}s_{11}+s_{12}d_{22}t_{22}=t_{22}c_{21}s_{12}.$$
It follows form Lemma \ref{lu} that $d_{22}=0$ and $d_{21}=c_{21}$.
\end{proof}

 Similarly, we have the following
 \begin{lemma}\label{lemma122122}
 For any $a_{12}\in \mathcal {R}_{12}$, $b_{21}\in \mathcal {R}_{21}$, and $c_{22}\in \mathcal {R}_{22}$, we have

  (i) $M(a_{12}+b_{21}+c_{22})=M(a_{12})+M(b_{21})+M(c_{22})$;

  (ii) $M^{*^{-1}}(a_{12}+b_{21}+c_{22})=M^{*^{-1}}(a_{12})+M^{*^{-1}}(b_{21})+M^{*^{-1}}(c_{22})$.
  \end{lemma}
 \begin{lemma}\label{lemma11122122}
  For any $a_{11}\in \mathcal {R}_{11}$, $b_{12}\in \mathcal {R}_{12}$, $c_{21}\in \mathcal {R}_{21}$, and $d_{22}\in \mathcal {R}_{22}$, the following hold.

(i)
$M(a_{11}+b_{12}+c_{21}+d_{22})=M(a_{11})+M(b_{12})+M(c_{21})+M(d_{22})$;

(ii)
$M^{*^{-1}}(a_{11}+b_{12}+c_{21}+d_{22})=M^{*^{-1}}(a_{11})+M^{*^{-1}}(b_{12})+M^{*^{-1}}(c_{21})+M^{*^{-1}}(d_{22})$.
 \end{lemma}
 \begin{proof}
 We pick $f=f_{11}+f_{12}+f_{21}+f_{22}\in \mathcal {R}$ such that $$M(f)=M(a_{11})+M(b_{12})+M(c_{21})+M(d_{22})=M(a_{11}+d_{22})+M(b_{12}+c_{21}).$$

 For any $t_{11}\in \mathcal {R}_{11}$ and $s_{12}\in \mathcal {R}_{12}$, we obtain
 \begin{eqnarray*}
 & &M^{*^{-1}}(t_{11}fs_{12}+s_{12}ft_{11})\\
 &=&M^{*^{-1}}(t_{11}(a_{11}+d_{22})s_{12}+s_{12}(a_{11}+d_{22})t_{11})\\
 & &+M^{*^{-1}}(t_{11}(b_{12}+c_{21})s_{12}+s_{12}(b_{12}+c_{21})t_{11}))\\
 &=&M^{*^{-1}}(t_{11}a_{11}s_{12})+M^{*^{-1}}(s_{12}c_{21}t_{11}) \\
 &=&M^{*^{-1}}(t_{11}a_{11}s_{12}+s_{12}c_{21}t_{11}).
 \end{eqnarray*}
 Note that in the last equality we apply Lemma \ref{lemma111221}. Then we get $t_{11}fs_{12}+s_{12}ft_{11}=t_{11}a_{11}s_{12}+s_{12}c_{21}t_{11}$. Furthermore, we have
 $$t_{11}f_{11}s_{12}+s_{12}f_{21}t_{11}=t_{11}a_{11}s_{12}+s_{12}c_{21}t_{11}.$$
  It follows form Lemma \ref{lu} that $f_{11}=a_{11}$ and $f_{21}=c_{21}$.

 We now show $f_{22}=d_{22}$ and $f_{12}=b_{12}$.   For any $t_{22}\in \mathcal {R}_{22}$ and $s_{21}\in \mathcal {R}_{21}$, we consider
 \begin{eqnarray*}
 & &M^{*^{-1}}(t_{22}fs_{21}+s_{21}ft_{22})\\
 &=&M^{*^{-1}}(t_{22}(a_{11}+d_{22})s_{21}+s_{21}(a_{11}+d_{22})t_{22})\\
 & &+M^{*^{-1}}(t_{22}(b_{12}+c_{21})s_{21}+s_{21}(b_{12}+c_{21})t_{22})\\
 &=&M^{*^{-1}}(t_{22}d_{22}s_{21})+M^{*^{-1}}(s_{21}b_{12}t_{22})\\
 &=&M^{*^{-1}}(t_{22}d_{22}s_{21}+s_{21}b_{12}t_{22}).
 \end{eqnarray*}
 Consequently, $$t_{22}fs_{21}+s_{21}ft_{22}=t_{22}d_{22}s_{21}+s_{21}b_{12}t_{22},$$ this implies that $t_{22}f_{22}s_{21}+s_{21}f_{12}t_{22}=t_{22}d_{22}s_{21}+s_{21}b_{12}t_{22}$. Thus $f_{22}=d_{22}$ and $f_{12}=b_{12}$.
 \end{proof}

  \noindent \textbf{Proof of Theorem \ref{theorem}} We first show that $M$ is additive. Let $a=a_{11}+a_{12}+a_{21}+a_{22}$ and $b=b_{11}+b_{12}+b_{21}+b_{22}$ be two arbitrary elements of $\mathcal {R}$. We have
 \begin{eqnarray*}
 & &M(a+b)\\
 &=&M((a_{11}+b_{11})+(a_{12}+b_{12})+(a_{21}+b_{21})+(a_{22}+b_{22}))\\
 &=&M(a_{11}+b_{11})+M(a_{12}+b_{12})+M(a_{21}+b_{21})+M(a_{22}+b_{22})\\
 &=&M(a_{11})+M(b_{11})+M(a_{12})+M(b_{12})+M(a_{21})+M(b_{21})+M(a_{22})+M(b_{22})\\
 &=&M(a_{11}+a_{12}+a_{21}+a_{22})+M(b_{11}+b_{12}+b_{21}+b_{22})\\
 &=&M(a)+M(b).
 \end{eqnarray*}
 That is, $M$ is additive.

 We now turn to prove that $M^*$ is additive. For any $x, y\in \mathcal {R}^{\prime}$, there exist $c=c_{11}+c_{12}+c_{21}+c_{22}$ and $d=d_{11}+d_{12}+d_{21}+d_{22}$ in $\mathcal {R}$ such that $c=M^*(x)+M^*(y)$ and $d=M^*(x+y)$.

 For arbitrary $t_{ij}\in \mathcal {R}_{ij}$ and $s_{kl}\in \mathcal {R}_{kl}$ ($1\leq i, j, k, l\leq 2$), using the additivity of $M$, we compute
 \begin{eqnarray*}
 & &M(t_{ij}cs_{kl}+s_{kl}ct_{ij})\\
 &=&M(t_{ij}(M^*(x)+M^*(y))s_{kl}+s_{kl}(M^*(x)+M^*(y))t_{ij})\\
 &=&M(t_{ij}M^*(x)s_{kl})+M(t_{ij}M^*(y)s_{kl})+M(s_{kl}M^*(x)t_{ij})+M(s_{kl}M^*(y)t_{ij})\\
 &=&M(t_{ij}M^*(x)s_{kl}+s_{kl}M^*(x)t_{ij})+M(t_{ij}M^*(y)s_{kl}+s_{kl}M^*(y)t_{ij})\\
 &=&M(t_{ij})xM(s_{kl})+M(s_{kl})xM(t_{ij})+M(t_{ij})yM(s_{kl})+M(s_{kl})yM(t_{ij})\\
 &=&M(t_{ij})(x+y)M(s_{kl})+M(s_{kl})(x+y)M(t_{ij})\\
 &=&M(t_{ij}M^*(x+y)s_{kl}+s_{kl}M^*(x+y)t_{ij})\\
 &=&M(t_{ij}ds_{kl}+s_{kl}dt_{ij}).
 \end{eqnarray*}
Therefore,
\begin{equation}\label{ii}
t_{ij}cs_{kl}+s_{kl}ct_{ij}=t_{ij}ds_{kl}+s_{kl}dt_{ij}.
\end{equation}

Letting $i=j=k=1$ and $l=2$ in equality (\ref{ii}), we get
 $$t_{11}c_{11}s_{12}+s_{12}c_{21}t_{11}=t_{11}d_{11}s_{12}+s_{12}d_{21}t_{11}.$$

It follows that $c_{11}=d_{11}$ and $c_{21}=d_{21}$.

We now set $i=j=k=2$ and $l=1$ in equality (\ref{ii}), then we obtain
$$t_{22}cs_{21}+s_{21}ct_{22}=t_{22}ds_{21}+s_{21}dt_{22},$$
that is,
$$t_{22}c_{22}s_{21}+s_{21}c_{12}t_{22}=t_{22}d_{22}s_{21}+s_{21}d_{12}t_{22}.$$
By Lemma \ref{lu}, we get $c_{22}=d_{22}$ and $c_{12}=d_{12}$.
Therefore, $c=d$, i.e., $M^*(x+y)=M^*(x)+M^*(y)$, which completes the proof.

For the case of Jordan triple elementary maps on prime rings we have the
following result.
\begin{corollary}
Let $\mathcal {R}$ be a $2$-torsion free unital prime  ring containing a
nontrivial idempotent $e_1$, and $\mathcal {R}^{\prime }$ be an
arbitrary ring.    Let $M\colon {\mathcal
R}\rightarrow {\mathcal R}^{\prime }$ and $M^*\colon {\mathcal
R}^{\prime }\rightarrow {\mathcal R}$ be two surjective maps such
that
\begin{displaymath}
 \left\{ \begin{array}{ll}
 M(aM^*(x)b+bM^*(x)a)=M(a)xM(b)+M(b)xM(a),\\
 M^*(xM(a)y+yM(a)x)=M^*(x)aM^*(y)+M^*(y)aM^*(x)
\end{array}\right.
\end{displaymath}
 for all $a, b\in \mathcal {R}$, $x, y\in {\mathcal R}^{\prime }$. Then both $M$ and $M^*$ are additive.
\end{corollary}
\begin{proof}
Since $\mathcal {R}$ is prime, it is easy to check that condition
(P) of Theorem \ref{theorem} holds true. Now the proof goes
directly.
\end{proof}

We complete this note by considering Jordan triple elementary maps on
standard operator algebras.

\begin{corollary}
Let $\mathcal {A}$ be a unital standard operator algebra on a Banach space
of dimension greater than $1$, and $\mathcal {R}$ an arbitrary ring.
Suppose that $M\colon {\mathcal A}\rightarrow {\mathcal R}$ and
$M^*\colon {\mathcal R}\rightarrow {\mathcal A}$ are surjective maps
such that
\begin{displaymath}
 \left\{ \begin{array}{ll}
 M(aM^*(x)b+bM^*(x)a)=M(a)xM(b)+M(b)xM(a),\\
 M^*(xM(a)y+yM(a)x)=M^*(x)aM^*(y)+M^*(y)aM^*(x)
\end{array}\right.
\end{displaymath}
 for all $a, b\in \mathcal {A}$, $x, y\in {\mathcal R}$. Then both $M$ and $M^*$ are additive.
\end{corollary}
 \bibliographystyle{amsplain}

\end{document}